\documentclass[11pt]{amsart}
   \usepackage[margin=3.9cm]{geometry}

\numberwithin{equation}{section}
\usepackage{amsmath,amsfonts,amsthm,amssymb,amscd, verbatim,graphicx,color,multirow,booktabs, caption,tikz,tikz-cd, mathdots,bm}
\usepackage{tikz-cd}
 \usepackage[pagebackref]{hyperref} 
\usetikzlibrary{positioning}

\newtheorem{theorem}{Theorem}[section]
\newtheorem{lemma}[theorem]{Lemma}

      \theoremstyle{definition}

     \theoremstyle{remark}
     \newtheorem{remark}[theorem]{Remark}

\newcommand{\Alt}{\mathop{\mathrm{Alt}}}

\newcommand{\dl}{\mathop{\mathrm{dl}}}

\newcommand{\Sol}{\mathop{\mathrm{Sol}}}
\newcommand{\SolG}{\mathop{\mathrm{Sol_G}}}

 \definecolor{mycolor}{rgb}{0.55,0.0,0.16}
  \definecolor{myred}{rgb}{0.75,0.0,0.16} 
  \definecolor{mygreen}{rgb}{0.0,0.4,0.16} 
  \definecolor{myviolet}{rgb}{1,0,1} 
   \definecolor{mypink}{rgb}{0.67,0,0.47}
	
	\newcommand{\MYhref}[3][black]{\href{#2}{\color{#1}{#3}}}% 
	
\hypersetup{
colorlinks=true,
linkcolor=true,
linktocpage=true,
pageanchor=true,
hyperindex=true
}

 \AtBeginDocument{
     \hypersetup{
  linkcolor=mycolor,
  urlcolor=mypink,
citecolor=mygreen
}
     }

\makeatletter
\@namedef{subjclassname@2020}{%
  \textup{2020} Mathematics Subject Classification}
\makeatother


\author[Luca Sabatini]{Luca Sabatini}
\address{\parbox{\linewidth}{Luca Sabatini, Mathematics Institute, University of Warwick}}
\email{luca.sabatini@warwick.ac.uk, sabatini.math@gmail.com}

\title[Arithmetic properties of solubilizers]{Arithmetic properties of solubilizers}

\begin{document} 

\maketitle 

\begin{abstract}
        Let $G$ be a finite group and $x \in G$.
        In 2013, Hai-Reuven proved that the cardinality of the centralizer of $x$
        divides the cardinality of the solubilizer of $x$.
        Settling a conjecture of Mousavi, Poozesh and Zamani,
        we show that the cardinality of the normalizer of $\langle x \rangle$ is also a divisor.
        We actually give direct short proofs of more general results.
        \end{abstract}

        %\tableofcontents

\vspace{0.2cm}
\section{Introduction} \label{sec1}

Let $G$ be a finite group and $x \in G$.
          The {\itshape solubilizer} $\Sol_G(x)$ is the set of the elements $g \in G$ such that $\langle x,g \rangle$ is soluble.
           In his MSc thesis in 2013, Hai-Reuven proved that the cardinality of the centralizer of $x$ divides the cardinality of $\Sol_G(x)$ \cite[Proposition~2.13]{HR13}.
         Unfortunately, his clever proof based on Cauchy--Frobenius counting lemma does not appear to be published in a refereed journal.

In more recent years, various aspects of solubilizers have been investigated \cite{ADM23,AFS26,ALMM20,BLN23,Luc24,MPZ25}.
In particular, Mousavi, Poozesh and Zamani \cite[Conjecture~3]{MPZ25}
conjectured in 2022 that the cardinality of $\Sol_G(x)$ should also be divisible by the cardinality of the \emph{normalizer} $N_G(x)$ of $\langle x \rangle$.
The reason this is interesting is that the natural approaches to prove it fail immediately.
The conjecture has been studied in \cite{ACSS25,Mou26}, mostly for finite simple groups.

          In this note, we expand the ideas in Hai-Reuven's thesis and prove the following much stronger statement by induction:
           
            \begin{theorem} \label{thMain}
           Let $G$ be a finite group, $x \in G$, and $K \leqslant G$.
           Then
           $$ |N_G(x) \cap K| \qquad \mbox{divides} \qquad |\SolG(x) \cap K| . $$
          \end{theorem} 

          The case $K=G$ establishes the aforementioned conjecture for the normalizers.
    The motivations behind this subgroup-relative version will be clarified later,
    but the key observation is that Hai-Reuven's methods fail badly in the induction step for the original statement.
    They also fail in proving Theorem~\ref{thMain}, but in a way that can be amended.

    With a different proof based on Hall's divisibility theorem, we also obtain a strengthening of the original result for centralizers.
    Writing $\dl(G)$ for the derived length of $G$ (and $\dl(G) = \infty$ if $G$ is not soluble),
    we set
          $$ \mathrm{Sol}_{G,\ell}(x) \> := \> \{ g \in G : \dl(\langle x,g \rangle) \leq \ell \} . $$
          
           \begin{theorem} \label{thHall}
Let $G$ be a finite group, $x\in G$, and $\ell \geq 1$. Then
\[
|C_G(x)| \qquad \mbox{divides} \qquad |\mathrm{Sol}_{G,\ell}(x)|.
\]
\end{theorem}

Note that, in a finite group, $\Sol_G(x) = \Sol_{G,\ell}(x)$ for some $\ell \geq 1$.
The same result holds for any variety of groups containing all abelian groups,
for example nilpotent groups of bounded class.

% it is actually true that if $G$ is perfect then $|N_G(x)|$ divides $\sharp \{g \in G: \langle x,g \rangle =G\}$.

          \vspace{0.1cm}
          \section{Preliminaries} 

    All groups in this paper are finite.
          In 1968, Thompson \cite{Tho68} proved that a finite group whose $2$-generated subgroups are soluble is soluble.
          A short proof was found in 1995 by Flavell \cite{Fla95}.
          
          The \emph{soluble radical} $R(G)$ is the largest normal soluble subgroup of $G$.
         For $x \in G$, we have
         $$ \frac{\langle x,R(G) \rangle}{R(G)} = \frac{\langle x \rangle R(G)}{R(G)} \cong \frac{\langle x \rangle}{\langle x \rangle \cap R(G)} . $$
         This implies that $\langle x,R(G)\rangle$ is soluble, and so
         $$ R(G) \> \subseteq \> \bigcap_{x \in G} \SolG(x) . $$
         The converse inclusion is a striking theorem of Guralnick, Kunyavskii, Plotkin and Shalev \cite{GKPS06},
         whose only known proof requires the Classification of Finite Simple Groups.
        We also refer the reader to the more recent results in \cite{HKMP21}.
        
          The \emph{normalizer} $N_G(x)$ is the set of elements $g \in G$ such that $\langle x \rangle \unlhd \langle x,g \rangle$. This is equivalent asking $\langle x \rangle = \langle x^g \rangle$.
        	For each $h \in N_G(x)$ we have $\langle x,h \rangle = \langle x \rangle \langle h \rangle$, and so
         $$ \langle x, h \rangle / \langle x \rangle \cong \langle h \rangle / (\langle x \rangle \cap \langle h \rangle) . $$
         It follows that $\langle x, h \rangle$ is metacyclic and $N_G(x) \subseteq \Sol_G(x)$.
        Since $N_G(x)$ and $\Sol_G(x)$ only depend on $\langle x \rangle$,
          for every $h \in N_G(x)$ we have $\Sol_G(x^h) = \Sol_G(x)$.
          In particular, $\Sol_G(x)$ is normalized by $N_G(x)$.

         The following famous lemma is the only prerequisite for the proof of Theorem~\ref{thMain}.
         See \cite{Neu79} for the history of this result, which is often misattributed to Burnside but appeared earlier in the work of Cauchy (1845) and Frobenius (1887).
        % We give a proof for the convenience of the reader.

          \begin{lemma}[Cauchy--Frobenius counting lemma] \label{lemCL}
        If a finite group $H$ acts on a finite set $A$, then the cardinality of the set
        $$ \{ ({a,h}) \in A \times H : a^h = a \} $$
        is a multiple of $|H|$. The ratio is the number of $H$-orbits on $A$.
        \end{lemma}
        \begin{comment}
        \begin{proof}
        The cardinality of the set above is $\sum_{a \in A} |C_H(a)|$.
         If $a_1,...,a_k \in A$ are representatives of the $H$-orbits on $A$, then
          $$ \sum_{a \in A} |C_H(a)| = \sum_{i=1}^k |(a_i)^H| |C_H(a_i)| = \sum_{i=1}^k |H| = k |H| , $$ 
          where the second equality follows by the orbit-stabilizer lemma.
        \end{proof}
        \end{comment}

         %\vspace{0.1cm}
        \section{Hai-Reuven's proof and an improvement}

    At first, one might think that the reason for $|C_G(x)|$ to divide $|\Sol_G(x)|$ should be structural. For example, one might hope that $\Sol_G(x)$ is a union of cosets of $C_G(x)$.
    It is easy to see that this is not the case, and something more subtle is underneath.

\begin{theorem}[Hai-Reuven, 2013] \label{thHR}
         Let $G$ be a finite group and $x \in G$. Then $|C_G(x)|$ divides $|\Sol_G(x)|$.
         \end{theorem}
         \begin{proof}
         We work by induction on $|G|$,
           and let $Z=Z(G)$ be the center of $G$.
         It is easy to see that $\frac{C_G(x)}{Z} \leqslant C_{G/Z}(Zx)$, and $\frac{\Sol_G(x)}{Z} = \Sol_{G/Z}(Zx)$.
         Now $|C_G(x)|$ divides $|\Sol_G(x)|$ if and only if $\frac{|C_G(x)|}{|Z|}$ divides $\frac{|\Sol_G(x)|}{|Z|}$,
         and so, by induction, we can assume $Z=1$.
         
         Let $A=\Sol_G(x)$, $H=C_G(x)$, and consider the action of $H$ on $A$ by conjugation.
        By Cauchy--Frobenius counting lemma, the quantity
        $$ \sum_{h \in H} |C_A(h)| $$
        is a multiple of $|H|$.
          For any $h \in H$ we have $x \in C_G(h)$, and
          $$ C_A(h) = C_G(h) \cap \SolG(x) = \mathrm{Sol}_{C_G(h)}(x) . $$
          We write
         $$ \sum_{h \in H} |C_A(h)| = |A| + \sum_{h \in H \setminus 1} |\mathrm{Sol}_{C_G(h)}(x)| , $$
          so it is sufficient to prove that $\sum_{h \in H \setminus 1} |\mathrm{Sol}_{C_G(h)}(x)|$ is a multiple of $|H|$.
          
          Let $h \in H \setminus 1$. 
          Since $h$ is noncentral in $G$, by induction
          we have $|\mathrm{Sol}_{C_G(h)}(x)| = n_h |C_{C_G(h)}(x)|$ for some integer $n_h$.
          We write
          $$ \sum_{h \in H \setminus 1} |\mathrm{Sol}_{C_G(h)}(x)| = \sum_{h \in H \setminus 1} n_h \> |C_{C_G(h)}(x)| . $$
          It is easy to see that $n_h = n_{h^g}$ for each $g \in H$.
          If $1,h_2,...,h_s \in H$ are representatives of the $H$-orbits on $H$ by conjugation, we have
          \begin{equation} \label{eq:3}
          \sum_{h \in H \setminus 1} |\mathrm{Sol}_{C_G(h)}(x)| = \sum_{i=2}^s n_{h_i} |(h_i)^H| |C_{C_G(h_i)}(x)| .
          \end{equation}
         Note that $C_{C_G(h_i)}(x) = C_G(h_i) \cap C_G(x) = C_H(h_i)$.
         By the orbit-stabilizer lemma, the right side of (\ref{eq:3}) is a multiple of $|H|$ as desired.
         \end{proof}

\subsection{A sharper result from Hall's divisibility theorem}

We use the following theorem of P.~Hall \cite{Hal36}:
the number of solutions of a system of equations in one unknown is
divisible by
\[
\gcd(|C|,n_1,n_2,\ldots),
\]
where $C$ is the centralizer of the set of coefficients, and $n_i$ is
the sum of the exponents of the unknown in the $i$-th equation.
See also \cite{GRV12} for a modern generalization.
We write $G^{(n)}$ for the $n$-th term of the derived series.

\begin{proof}[Proof of Theorem~\ref{thHall}]
Let $F=F(X,Y)$ be the free group on $X$ and $Y$,
and set $A:=\mathrm{Sol}_{G,\ell}(x)$.
For each $y \in G\setminus A$, there is a word $w_y\in F^{(\ell)}$
such that $w_y(x,y)\neq 1$.
It is easy to check that $A$ is the set of solutions in $G$ of the system of equations
\[
w_y(x,Y)=1,
\qquad y\in G\setminus A.
\]
Since $F^{(\ell)} \subseteq F^{(1)}$, the sum of the exponents of $Y$ in each equation is
zero. By Hall's theorem \cite[Theorem~II]{Hal36}, $|A|$ is divisible by the order of the centralizer of all the
coefficients occurring in the system.
All coefficients belong to $\langle x\rangle$, so their centralizer
contains $C_G(x)$, and the proof follows.
\end{proof}

The strategy used to prove Theorem~\ref{thHall} is not suitable for handling questions with normalizers.

\begin{remark}
Since $N_G(x) \subseteq \Sol_{G,2}(x)$,
looking at Theorem~\ref{thHall}, a tempting approach for the normalizer-solubilizer conjecture would be to first show that
          $|N_G(x)|$ divides $|\Sol_{G,2}(x)|$, and then that $|\Sol_{G,\ell}(x)|$ divides $|\Sol_{G,\ell+1}(x)|$ for all $\ell \geq 2$.
          However, there are examples in $\Alt(6)$ where $|N_G(x)|$ does not divide $|\Sol_{G,2}(x)|$ (but still divides $|\Sol_G(x)|$).
          \end{remark}

%\vspace{0.1cm}
\section{Normalizers and the relative version}

If $h \in C_G(x)$ then $x \in C_G(h)$,
while in general $h \in N_G(x)$ does not imply $x \in N_G(h)$.
This is the crucial point where Hai-Reuven's proof fails in extending to normalizers.
To overcome this problem, we introduce an auxiliary subgroup $K$ (in the end will be $K=G$) and argue by induction on $|K|$.
Noting that $N_G(x) \cap K$ acts on $\Sol_G(x) \cap K$ by conjugation,
the only obstruction to applying induction to the fixed-point terms is given by
the nontrivial elements in $N_G(x) \cap Z(K)$,
and this is the main technical issue we are left to deal with.

\begin{figure}[htbp]
\centering
\begin{tikzpicture}[scale=0.7]

  % fill Sol_G(x) \cap K in light gray
  \begin{scope}
    \clip (-1,0.5) ellipse (2cm and 3cm);          % Sol_G(x)
    \fill[gray!30] (2,0.5) ellipse (2.5cm and 1cm); % K
  \end{scope}

  % fill N_G(x) \cap K in darker gray
  \begin{scope}
    \clip (-1,0.5) ellipse (1cm and 1.5cm);         % N_G(x)
    \fill[gray!60] (2,0.5) ellipse (2.5cm and 1cm); % K
  \end{scope}

  % draw the sets
  \draw (-1,0.5) ellipse (2cm and 3cm);
  \draw (-1,0.5) ellipse (1cm and 1.5cm);
  \draw (-1,0.5) ellipse (0.02cm and 0.02cm);

  \draw (2,0.5) ellipse (2.5cm and 1cm);

  % labels
  \node at (-3,2.2) {\small $\Sol_G(x)$};
  \node at (-2,1.2) {\small $N_G(x)$};
  \node at (-1,0.2) {\small $x$};
  \node at (4,2) {\small $K$};

\end{tikzpicture}
\caption*{Sets and subgroups involved in Theorem~\ref{thMain}.}
\end{figure}
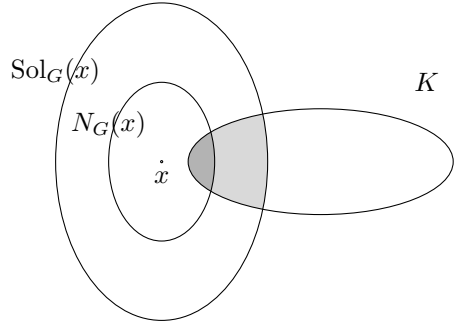

\begin{proof}[Proof of Theorem~\ref{thMain}]
We set
\[
H:=N_G(x)\cap K,
\qquad
A:=\SolG(x)\cap K,
\]
and argue by induction on $|K|$.
Since $N_G(x)$ normalizes $\Sol_G(x)$, the group $H$ acts on $A$ by conjugation.

Let $T \unlhd H$ be maximal with respect to inclusion with the properties that
\[
T\unlhd K \qquad \mbox{and} \qquad
\langle x,a,T\rangle \text{ is soluble for every } a\in A.
\]

We claim that
\begin{equation} \label{eqClaim}
    (H/T)\cap Z(K/T)=1.
\end{equation}
Let $D$ be the inverse image of \((H/T)\cap Z(K/T)\) in $H$.
Then $D\unlhd K$. Suppose that $D>T$. Fix $a\in A$ and put
\[
U:=\langle x,a,T\rangle ,
\]
so that $U$ is soluble by the choice of $T$.
Moreover,
\[
U\unlhd UD=\langle x,a,D\rangle
\]
and $UD/U$ is abelian. Thus $\langle x,a,D\rangle$ is soluble for
every $a\in A$, contradicting the maximality of $T$.
This proves (\ref{eqClaim}).

For every $a\in A$ and $t\in T$ we have that
\(
\langle x,at\rangle \leqslant \langle x,a,T\rangle
\) is soluble,
and therefore $at\in A$. Hence $A$ is a union of $T$-cosets.
Write $H:=H/T$, and $A:=A/T$ for the set of $T$-cosets contained in $A$.
Note that
\begin{equation} \label{eqQuot}
    |\overline H|=\frac{|H|}{|T|}
\qquad\text{and}\qquad
|\overline A|=\frac{|A|}{|T|}.
\end{equation}

Let $1\neq\overline h\in\overline H$, and let $L$ be the inverse image
in $K$ of $C_{K/T}(\overline h)$. By (\ref{eqClaim}), $L<K$.
Moreover,
\[
C_{\overline H}(\overline h)
   =\frac{H\cap L}{T},
\qquad
C_{\overline A}(\overline h)
   =\frac{A\cap L}{T}.
\]

By induction applied to $L$,
$|N_G(x)\cap L|$ divides $|\SolG(x)\cap L|$.
Since $L \subseteq K$, this says that
$|H\cap L|$ divides $|A\cap L|$, and we obtain
\begin{equation} \label{eqInd}
    |C_{\overline H}(\overline h)|
   \qquad \mbox{divides} \qquad
|C_{\overline A}(\overline h)|.
\end{equation}

By Cauchy--Frobenius counting lemma applied to the action of $\overline H$ on
$\overline A$ by conjugation, we have that
\begin{equation} \label{eqBL}
    \sum_{\overline h\in\overline H}
|C_{\overline A}(\overline h)|
\qquad \mbox{is a multiple of} \qquad |\overline H| .
\end{equation}
By the orbit-stabilizer lemma, for every nonidentity conjugacy class
$\overline h^{\,\overline H}$, the contribution to this sum is
\[
|\overline h^{\,\overline H}|\,
|C_{\overline A}(\overline h)|
=
\frac{|\overline H|}
     {|C_{\overline H}(\overline h)|}
|C_{\overline A}(\overline h)|,
\]
which is a multiple of $|\overline H|$ by (\ref{eqInd}).
Therefore all nonidentity conjugacy classes contribute a multiple of
$|\overline H|$ to (\ref{eqBL}). Hence the identity contribution
does as well, meaning precisely that
$|\overline H|$ divides $|\overline A|$.
The proof follows by (\ref{eqQuot}).
\end{proof}

\vspace{0.2cm} \noindent
 {\bfseries Acknowledgments.}
I heard about the normalizer-solubilizer conjecture during
a talk of Costantino Delizia in Paestum in July 2023.
The same day I had useful conversations with Carmine Monetta,
who pointed me to Hai-Reuven's thesis. My thanks go to both of them.
I would also like to apologize to Derek Holt, whom I urged during several lunches to search for (nonexistent) counterexamples to Theorem~\ref{thMain}.
Finally, I thank Juan Mart\'inez Madrid for telling me about the history of Cauchy--Frobenius counting lemma.

\vspace{0.2cm} \noindent
 {\bfseries AI statement:}
Already in 2023 I had realized that Theorem~\ref{thMain} was the right statement to prove. In the following three years I obtained several reductions, as well as Theorem~\ref{thHall}.
The final proof of Theorem~\ref{thMain} was completed with AI assistance. Specifically, the strategy of introducing the subgroup $T$ and proving (\ref{eqClaim}) was suggested by the AI model. I personally wrote the complete first draft of the paper and AI was used again only for final proofreading.

   \vspace{0.2cm}


\vspace{0.1cm}
\begin{thebibliography}{Z} 

\bibitem{ACSS25} B. Akbari, J. Chuharski, V. Sharan and Z. Slonim,
\textit{Characterization of solubilizers of elements in minimal simple groups},
Comm. Algebra {\bfseries 53} (2025), 1995--2017.

\bibitem{ADM23} B. Akbari, C. Delizia and C. Monetta,
\textit{On the solubilizer of an element in a finite group},
Mediterr. J. Math. {\bfseries 20} (2023), art. 135.

\bibitem{AFS26} B. Akbari, T. Foguel and J. Schmidt,
\textit{Solvabilizer numbers of finite groups},
to appear in Int. J. Algebra Comput. (2026).

\bibitem{ALMM20} B. Akbari, M. L. Lewis, J. Mirzajani and A. R. Moghaddamfar,
\textit{The solubility graph associated with a finite group},
Int. J. Algebra Comput. {\bfseries 30} (2020), 1555--1564.

\bibitem{BLN23} T.C. Burness, A. Lucchini and D. Nemmi,
\textit{On the soluble graph of a finite group},
J. Combin. Theory Ser. A {\bfseries 194} (2023), 105708.

\bibitem{Fla95} P. Flavell,
\textit{Finite groups in which every two elements generate a soluble subgroup},
Invent. Math. {\bfseries 121} (1995), 279--285.

%\bibitem{Fla02} P. Flavell,
%\textit{A characterisation of $F_2(G)$},
%J. Algebra {\bfseries 255} (2002), 271--287.

\bibitem{GRV12} C. Gordon and F. Rodriguez-Villegas,
\textit{On the divisibility of $\sharp \mathrm{Hom}(\Gamma,G)$ by $|G|$},
J. Algebra {\bfseries 350} (2012), 300--307.

\bibitem{GKPS06} R. Guralnick, B. Kunyavski\u{\i}, E. Plotkin and A. Shalev,
\textit{Thompson-like characterizations of the solvable radical},
J. Algebra {\bfseries 300} (2006), 363--375.

\bibitem{HR13} D. Hai-Reuven,
\textit{Non-solvable graph of a finite group and solvabilizers},
 MSc thesis (supervisor B.~Kunyavski\u{\i}),
 available at \texttt{\MYhref[mypink]{https://arxiv.org/pdf/1307.2924}{arxiv.org/pdf/1307.2924}} (2013).

\bibitem{Hal36} P. Hall,
\textit{On a theorem of Frobenius},
Proc. Lond. Math. Soc. {\bfseries 40} (1936), 468--501.

\bibitem{HKMP21} P. Hauck, L. S. Kazarin, A. Mart\'inez-Pastor and M. D. P\'erez-Ramos,
\textit{Thompson-like characterization of solubility for products of finite groups},
Ann. Mat. Pura Appl. {\bfseries 200} (2021), 337--362.

%\bibitem{KM14} A.A. Klyachko and A.A. Mkrtchyan,
%\textit{How many tuples of group elements have a given property?},
%Int. J. Algebra Comput. {\bfseries 24} (2014), 413--428.

\bibitem{Luc24} A. Lucchini,
\textit{Solubilizers in profinite groups},
J. Algebra {\bfseries 647} (2024), 619--632.

\bibitem{Mou26} H. Mousavi,
\textit{On the normalizer-solubilizer conjecture},
Int. J. Group Theory {\bfseries 15} (2026), 107--122.

\bibitem{MPZ25} H. Mousavi, M. Poozesh and Y. Zamani,
\textit{The impact of the solubilizer of an element on the structure of a finite group},
Ricerche Mat. {\bfseries 74} (2025), 603--612.

\bibitem{Neu79} P.M. Neumann,
\textit{A lemma that is not Burnside's},
Math. Sci. {\bfseries 4} (1979), 133--141.

\bibitem{Tho68} J.G. Thompson,
\textit{Nonsolvable finite groups all of whose local subgroups are solvable},
Bull. Amer. Math. Soc. {\bfseries 74} (1968), 383--437.

   \end{thebibliography}
\end{document}